\documentclass{amsart}

\usepackage{amsmath,amsthm,amsfonts,amssymb,amscd,latexsym}
\usepackage[all]{xy}

\newtheorem{thm}{Theorem}[section]
\newtheorem{lem}[thm]{Lemma}
\newtheorem{pro}[thm]{Proposition}
\newtheorem{cor}[thm]{Corollary}

\numberwithin{equation}{section}

\newcommand{\ad}{\mathrm{ad}}

\newcommand{\Ker}{\mathrm{Ker}}
\newcommand{\im}{\mathrm{Im}}
\newcommand{\coker}{\mathrm{Coker}}
\newcommand{\Hom}{\mathrm{Hom}}

\newcommand{\Der}{\mathrm{Der}}
\newcommand{\ind}{\mathrm{Ind}}
\newcommand{\mt}{\mathrm{MT}}
\newcommand{\irr}{\mathrm{Irr}}
\newcommand{\gr}{\mathrm{gr}}
\newcommand{\dx}{\mathrm{dx}}
\newcommand{\dy}{\mathrm{dy}}

\newcommand{\F}{\mathbb{F}}
\newcommand{\Z}{\mathbb{Z}}

\newcommand{\boD}{\mathbf{D}}
\newcommand{\boY}{\mathbf{Y}}
\newcommand{\boX}{\mathbf{X}}

\newcommand{\caP}{\mathcal{P}}
\newcommand{\caF}{\mathcal{F}}
\newcommand{\caI}{\mathcal{I}}
\newcommand{\caD}{\mathcal{D}}
\newcommand{\caO}{\mathcal{O}}

\newcommand{\fW}{\mathfrak{W}}
\newcommand{\fT}{\mathfrak{T}}
\newcommand{\fB}{\mathfrak{B}}

\newcommand{\fsl}{\mathfrak{sl}}
\newcommand{\fF}{\mathfrak{F}}
\newcommand{\fu}{\mathfrak{u}}

\newcommand{\der}{\partial}
\newcommand{\uone}{\underline{1}}
\newcommand{\duone}{\underline{\underline{1}}}

\begin{document}

%%%%%%%%%%%%%%%%%%%%%%%%%%%%%%%%%%%%%%%%%%%%%%%%%%%%%%%%%%%%%%%%%%%%%%%%%%%%%%%%%%%%%%%%%%%%%%%%%%%

\title[Restricted Lie algebras with maximal $0$-PIM]{Restricted Lie algebras with maximal $0$-PIM}

\author{J\"org Feldvoss}
\address{Department of Mathematics and Statistics, University of South Alabama,
Mobile, AL 36688--0002, USA}
\email{jfeldvoss@southalabama.edu}

\author{Salvatore Siciliano}
\address{Dipartimento di Matematica e Fisica ``Ennio de Giorgi", Universit\`a del
Salento, Via Provinciale Lecce-Arnesano, I-73100 Lecce, Italy}
\email{salvatore.siciliano@unisalento.it}

\author{Thomas Weigel}
\address{Dipartimento di Matematica e Applicazioni, Universit\`a degli Studi di
Milano-Bicocca, Via Roberto Cozzi, No.\ 53, I-20125 Milano, Italy}
\email{thomas.weigel@unimib.it}



%%%%%%%%%%%%%%%%%%%%%%%%%%%%%%%%%%%%%%%%%%%%%%%%%%%%%%%%%%%%%%%%%%%%%%%%%%%%%%%%%%%%%%%%%%%%%%%%%%%

\begin{abstract}
In this paper we show that the projective cover of the trivial irreducible
module of a finite-dimensional solvable restricted Lie algebra is induced
from the one-dimensional trivial module of a maximal torus. As a consequence,
we obtain that the number of the isomorphism classes of irreducible modules
with a fixed $p$-character for a finite-dimensional solvable restricted Lie
algebra $L$ is bounded above by $p^{\mt(L)}$, where $\mt(L)$ denotes
the largest dimension of a torus in $L$. Finally, we prove that in characteristic
$p>3$ the projective cover of the trivial irreducible $L$-module is only induced
from the one-dimensional trivial module of a torus of maximal dimension if
$L$ is solvable.
\end{abstract}

%%%%%%%%%%%%%%%%%%%%%%%%%%%%%%%%%%%%%%%%%%%%%%%%%%%%%%%%%%%%%%%%%%%%%%%%%%%%%%%%%%%%%%%%%%%%%%%%%%%
      
\date{July 7, 2014}
          
\maketitle

%%%%%%%%%%%%%%%%%%%%%%%%%%%%%%%%%%%%%%%%%%%%%%%%%%%%%%%%%%%%%%%%%%%%%%%%%%%%%%%%%%%%%%%%%%%%%%%%%%%

\section*{Introduction} 

%%%%%%%%%%%%%%%%%%%%%%%%%%%%%%%%%%%%%%%%%%%%%%%%%%%%%%%%%%%%%%%%%%%%%%%%%%%%%%%%%%%%%%%%%%%%%%%%%%%

For a fixed prime number $p$ Wolfgang Willems considered the class $\caP_0(p)$ of
all finite groups $G$ for which the dimension of the projective cover of the one-dimensional
trivial $\F G$-module over a field $\F$ of characteristic $p$ has minimal dimension
and compared $\caP_0(p)$ to the class of all finite groups having a $p$-complement.
In particular, every $p$-solvable group belongs to $\caP_0(p)$, but the converse is
not true (see \cite[Section~4]{Wil}). More recently, Gunter Malle and the third author of
this paper classified all finite non-abelian simple groups belonging to $\caP_0(p)$
for a fixed prime number $p$ by using the classification of finite simple groups (see
\cite[Theorem A]{MW}). As a consequence, they obtained that a finite group $G$
is solvable (i.e., $G$ is $p$-solvable for every prime number $p$) if, and only if,
$G\in\caP_0(p)$ for every prime number $p$ (see \cite[Corollary~B]{MW}).

In this paper we investigate the analogous question for finite-dimensional restricted
Lie algebras over a field of prime characteristic. It turns out that for a restricted Lie
algebra there is a canonical upper bound for the dimension of the projective cover
of its one-dimensional trivial module (see Proposition \ref{upbd}). We say that a
finite-dimensional restricted Lie algebra has maximal $0$-PIM if this maximal possible
dimension is attained (see Section 2). One main goal of the paper is then to classify
all finite-dimensional restricted Lie algebras having maximal $0$-PIM. We prove that
a finite-dimensional restricted Lie algebra over a field of characteristic $p>3$ has
maximal $0$-PIM if, and only if, it is solvable (see Theorem \ref{max0PIM}). A
main ingredient of the proof is the classification of finite-dimensional simple Lie algebras
of absolute toral rank two over an algebraically closed field of characteristic $p>3$
(see \cite{StrII}). We show by an example that in characteristic two there exists a
finite-dimensional non-solvable restricted Lie algebra having maximal $0$-PIM. In
characteristic three we do not know of such a counterexample, but according to
the lack of a classification of finite-dimensional simple Lie algebras of absolute toral
rank two in this case, at the moment it is not clear whether our result holds in characteristic
three.

In the first section we collect several useful results for projective covers of modules
over reduced enveloping algebras that will be needed later in the paper. Some of these
results were already known in special cases, but for the convenience of the reader we
treat them here in one place. The most important one is Theorem \ref{fong}, which is
the Lie-theoretic analogue of a result for finite modular group algebras due to Wolfgang
Willems (see \cite[Lemma 2.6]{Wil} or also \cite[Lemma VII.14.2]{HB}), and will be crucial
in the proof of the main result of the next section, namely, that every finite-dimensional
solvable restricted Lie algebra has maximal $0$-PIM (see Theorem \ref{solvmax0PIM}).
In Section 3 this result is employed to establish an upper bound for the number of the isomorphism
classes of irreducible modules with a fixed $p$-character for a solvable restricted Lie algebra
(see Theorem \ref{numirr}). This generalizes the known results for nilpotent (see \cite[Satz
6]{Str1}) and supersolvable restricted Lie algebras (see \cite[Theorem 4]{Fe5}). For quite some
time the first author has conjectured that this bound holds for {\em every\/} finite-dimensional
restricted Lie algebra (see also \cite[Section 10, Conjecture]{Hum} for simple Lie algebras
of classical type). In the fourth section we prove a general result on the compatibility of
induction for a filtered restricted Lie algebra of depth one with the associated graded restricted
Lie algebra which might be of general interest. In particular, the irreducibility of the induced
module for the associated graded restricted Lie algebra implies the irreducibility of the
corresponding induced module for the filtered restricted Lie algebra (see Theorem \ref{grind}). 
Section 5 discusses some non-graded Hamiltonian Lie algebras and their representations.
Here the irreducibility of certain induced modules is obtained from Theorem \ref{grind} and
the known corresponding result for the associated graded restricted Lie algebra (see
Theorem \ref{irr}). The last section is devoted to a proof of the converse of Theorem
\ref{solvmax0PIM} in characteristic $p>3$ and a counterexample to this result in characteristic
two. In characteristic three this converse seems to be open, and we hope to come back
to this on another occasion.

In the following we briefly recall some of the notation that will be used in this paper. By
$\Z_{\ge 0}$ we will denote the set of all non-negative integers. Let $\langle X\rangle_{\F}$
denote the $\F$-subspace of a vector space $V$ over a field $\F$ spanned by a subset
$X$ of $V$. For a subset $X$ of a restricted Lie algebra $L$ we denote by $\langle X
\rangle_p$ the $p$-subalgebra of $L$ generated by $X$. Finally, $[L,L]$ or $L^{(1)}$
will denote the derived subalgebra of a Lie algebra $L$. For more notation and some
well-known results from the structure and representation theory of modular Lie algebras
we refer the reader to Chapters 1 -- 5 in \cite{SF} as well as Chapters 1, 4, and
7 in \cite{StrI}.

%%%%%%%%%%%%%%%%%%%%%%%%%%%%%%%%%%%%%%%%%%%%%%%%%%%%%%%%%%%%%%%%%%%%%%%%%%%%%%%%%%%%%%%%%%%%%%%%%%%

\section{Projective covers of modules with a $p$-character}

%%%%%%%%%%%%%%%%%%%%%%%%%%%%%%%%%%%%%%%%%%%%%%%%%%%%%%%%%%%%%%%%%%%%%%%%%%%%%%%%%%%%%%%%%%%%%%%%%%%

Let $A$ be a finite-dimensional unital associative algebra, and let $M$ be
a (unital left) $A$-module. Recall that a projective module $P_A(M)$ is a
{\it projective cover\/} of $M$, if there exists an $A$-module epimorphism
$\pi_M$ from $P_A(M)$ onto $M$ such that the kernel of $\pi_M$ is contained
in the radical of $P_A(M)$. If projective covers exist, then they are unique
up to isomorphism (see also the remark after (PC) below). It is well known
that projective covers of finite-dimensional modules over finite-dimensional
associative algebras always exist and are again finite-dimensional. Moreover,
every projective indecomposable $A$-module is isomorphic to the projective
cover of its irreducible head. In this way one obtains a bijection between the
isomorphism classes of the projective indecomposable $A$-modules and the
isomorphism classes of the irreducible $A$-modules.

In the sequel we will need the following universal property of the pair $(P_A
(M),\pi_M)$ which is an immediate consequence of Nakayama's lemma:
\begin{itemize}
\item[(PC)]
If $P$ is a projective $A$-module and $\pi$ is an $A$-module epimorphism
from $P$ onto $M$, then every $A$-module homomorphism $\eta$ from $P$
into $P_A(M)$ with $\pi_M\circ\eta = \pi$ is an epimorphism.
\end{itemize}
\vspace{.25cm}

\noindent {\bf Remark.} Since $P$ is projective and $\pi_M$ is an epimorphism,
there always exists an $A$-module homomorphism $\eta$ from $P$ to $P_A(M)$
such that $\pi_M\circ\eta = \pi$. In particular, it follows from (PC) that projective
covers are unique up to isomorphism.
\vspace{.3cm}

Let $L$ be a finite-dimensional restricted Lie algebra over a field of prime
characteristic $p$, and let $\chi$ be a linear form on $L$. Moreover, let $u(L,\chi)$
denote the $\chi$-reduced universal enveloping algebra of $L$ (see \cite[\S1.3]{VK}
or \cite[p.\ 212]{SF}), and let $P_L(M):=P_{u(L,\chi)} (M)$ denote the projective
cover of the $u(L,\chi)$-module $M$. If $H$ is a $p$-subalgebra of $L$, and $V$
is an $H$-module with $p$-character $\chi_{\vert H}$, then we set $\ind_H^L(V,\chi)
:=u(L,\chi)\otimes_{u(H,\chi_{\vert H})}V$ (see \cite[\S1.3]{VK} or \cite[Section
5.6]{SF} for the usual properties of induction).

As we could not find a reference for our first result, which will be needed in the
proof of Theorem \ref{max0PIM}, it is included here for completeness.

\begin{pro}\label{ind}
Let $L$ be a finite-dimensional restricted Lie algebra over a field of prime characteristic
$p$, let $\chi$ be a linear form on $L$, let $H$ be a $p$-subalgebra of $L$, and
let $M$ be a finite-dimensional $L$-module with $p$-character $\chi$. Then $P_L
(M)$ is a direct summand of $\ind_H^L(P_H(M_{\vert H}),\chi)$.
\end{pro}

\begin{proof}
Set $P:=\ind_H^L(P_H(M_{\vert H}),\chi)$. Since induction preserves projectivity,
$P$ is a projective $u(L,\chi)$-module. As $P_H(M_{\vert H})$ is the projective
cover of $M_{\vert H}$, there is an $H$-module epimorphism $\pi$ from $P_H
(M_{\vert H})$ onto $M_{\vert H}$. According to \cite[Theorem~5.6.3]{SF},
there exists an $L$-module epimorphism from $P$ onto $M$. Then we obtain
from (PC) and the subsequent remark that $P_L(M)$ is a homomorphic image of
$P$, and thus also a direct summand of $P$.
\end{proof}

Let $$\mt(L):=\max\{\dim_\F T\mid T\mbox{ is a torus of }L\}$$ denote the
{\em largest dimension\/} of a torus in $L$ (see \cite[Notation 1.2.5]{StrI}).
The next result is a consequence of Proposition \ref{ind} and the semisimplicity
of reduced universal enveloping algebras of tori.
%The argument of the next result is the same as the one at the beginning of
%the proof of \cite[Proposition 1]{Fe4}.

\begin{pro}\label{upbd}
Let $L$ be a finite-dimensional restricted Lie algebra over a field $\F$ of prime
characteristic $p$, let $\chi$ be a linear form on $L$, and let $M$ be a
finite-dimensional $L$-module with $p$-character $\chi$. Then $P_L(M)$
is a direct summand of $\ind_T^L(M_{\vert T},\chi)$ for any torus $T$ of $L$.
In particular, $$\dim_\F P_L(M)\le(\dim_\F M)\cdot p^{\dim_\F L-\mt(L)}\,.$$
\end{pro}

\begin{proof}
Let $T$ be any torus of $L$. By virtue of \cite[Lemma 3.2]{Fe3}, $M_{\vert T}$
is a projective $u(T,\chi_{\vert T})$-module, and it follows from Proposition \ref{ind}
that $P_L(M)$ is a direct summand of $\ind_T^L(M_{\vert T},\chi)$.

Assume now in addition that $T $ is of maximal dimension in $L$, i.e., $\dim_\F
T=\mt(L)$. Then it follows from the first part and \cite[Proposition 5.6.2]{SF}
that $$\dim_\F P_L(M)\le\dim_\F\ind_T^L(M_{\vert T},\chi)=(\dim_\F M)\cdot
p^{\dim_\F L/T}=(\dim_\F M)\cdot p^{\dim_\F L-\mt(L)}\,.$$
\end{proof}

For later use we record the following result, which is a consequence of the first
part of Proposition \ref{upbd} and \cite[Proposition 5.6.2]{SF}:

\begin{cor}\label{proj}
Let $L$ be a finite-dimensional restricted Lie algebra over a field $\F$ of prime
characteristic $p$, and let $\chi$ be a linear form on $L$. If $M$ is a finite-dimensional
$L$-module with $p$-character $\chi$ such that $$\dim_\F P_L(M)=(\dim_\F M)
\cdot p^{\dim_\F L-\mt(L)}\,,$$ then $P_L(M)\cong\ind_T^L(M_{\vert T},\chi)$
holds for any torus $T$ of $L$ of maximal dimension.
\end{cor}

The argument of our next result, which will be needed in Section 3,  is the same
as in the proof of \cite[Theorem 2]{Fe2}, but for the convenience of the reader
it is included here. We will denote the trivial irreducible $L$-module of a restricted
Lie algebra $L$ over a field $\F$ also just by $\F$.

\begin{pro}\label{reciproc}
Let $L$ be a finite-dimensional restricted Lie algebra over a field $\F$ of prime
characteristic $p$, let $\chi$ be a linear form on $L$, and let $M$ be a non-zero
finite-dimensional $L$-module with $p$-character $\chi$. Then $P_L(\F)$ is a direct
summand of $P_L(M)\otimes M^*$.
\end{pro}

\begin{proof}
Set $P:=P_L(M)\otimes M^*$. According to \cite[Theorem 5.2.7(2)]{SF} and
\cite[Lemma~2.3]{Fe1}, $P$ is a projective $u(L,0)$-module. As $P_L(M)$ is the
projective cover of $M$, there is an $L$-module epimorphism $\pi_M$ from $P_L
(M)$ onto $M$. Tensoring $\pi_M$ with the identity transformation of $M^*$
yields an $L$-module epimorphism from $P$ onto $M\otimes M^*$ which in turn
has $\F$ as an epimorphic image. Now we can proceed as in the proof of Proposition
\ref{ind}.
\end{proof}

The following dimension formula is the Lie-theoretic analogue of a result for finite
modular group algebras due to Wolfgang Willems (see \cite[Lemma 2.6]{Wil} or
\cite[Lemma VII.14.2]{HB}), and will be crucial for the proof of our first main result
in the next section.

\begin{thm}\label{fong}
Let $L$ be a finite-dimensional restricted Lie algebra over a field $\F$ of prime
characteristic $p$, let $\chi$ be a linear form on $L$, and let $I$ be a $p$-ideal
of $L$. Then $$\dim_\F P_L(S)=\dim_\F P_I(\F)\cdot\dim_\F P_{L/I}(S)$$ holds
for every irreducible $L$-module $S$ with $p$-character $\chi$ such that $I\cdot
S=0$.
\end{thm}

\begin{proof}
As both the restriction functor from $L$ to $I$ and the induction functor from
$I$ to $L$ are exact additive covariant functors that map projectives to projectives,
their composition $\caF:=\ind_I^L(-,\chi)_{\vert I}$ has the same properties. It
follows from the generalized Cartan-Weyl identity (see \cite[Lemma 5.7.1]{SF})
that $I$ annihilates the module $\caF(\F)$ which is therefore isomorphic to $\dim_\F
u(L/I,\chi)$ copies of the trivial irreducible $I$-module. The exactness of $\caF$
and (PC) implies that $\caF(P_I(\F))$ contains a direct summand that is isomorphic
to $\dim_\F u(L/I,\chi)$ copies of $P_I(\F)$. Since both modules have the same
dimension, they are isomorphic.

As $S$ is a trivial $I$-module, there exists a non-zero $I$-module homomorphism
from $P_I(\F)$ to $S_{\vert I}$, and one obtains from the universal property of
induced modules (see \cite[Theorem 5.6.3]{SF}) and the irreducibility of $S$ that
there exists an $L$-module epimorphism from $\ind_I^L(P_I(\F),\chi)$ onto $S$.
But the former module is projective and thus it follows from (PC) that $P_L(S)$ is
an epimorphic image of $\ind_I^L(P_I(\F),\chi)$. Since the former is also projective,
it is a direct summand of the latter. Then after restriction to $I$ the Krull-Remak-Schmidt
Theorem in conjunction with the isomorphism $\caF(P_I(\F))\cong P_I(\F)^{\oplus\dim_\F
u(L/I,\chi)}$ implies that $P_L(S)_{\vert I}\cong P_I(\F)^{\oplus e}$ for some positive
integer $e$. In particular, we obtain that $\dim_\F P_L(S)=e\cdot\dim_\F P_I(\F)$.

Let $\caD$ denote the left adjoint functor of the inflation functor $\caI$ from $L/I$
to $L$. It can be shown that $\caD$ is just the coinvariants functor so that $\caD
(M)=M/IM$ for every $u(L,\chi)$-module $M$. Since $\caI$ is obviously exact, it follows
that $\caD$ maps projectives to projectives (see \cite[Proposition 2.3.10]{Wei}). For
any irreducible $u(L/I,\chi)$-module $M$ one has
\begin{eqnarray*}
\Hom_{L/I}(\caD(P_L(S)),M) & \cong & \Hom_L(P_L(S),\caI(M))\cong\Hom_L(S,\caI(M))\\
& \cong & \Hom_{L/I}(S,M)\cong\Hom_{L/I}(P_{L/I}(S),M)\,.
\end{eqnarray*}
Hence (PC) implies that $\caD(P_L(S))$ and $P_{L/I}(S)$ are isomorphic. As $P_I(\F)/IP_I
(\F)$ is one-dimensional, one concludes from $P_L(S)_{\vert I}\cong P_I(\F)^{\oplus e}$ that
$$e=\dim_\F P_L(S)/IP_L(S)=\dim_\F\caD(P_L(S))=\dim_\F P_{L/I}(S)\,.$$ Consequently,
the assertion follows from the last equality of the previous paragraph.
\end{proof}

\noindent {\bf Remark.} The above proof would also work in the case of finite-dimensional
modular group algebras. In particular, this provides an alternative proof of Willems' result
\cite[Lemma 2.6]{Wil}.
\vspace{.3cm}

%%%%%%%%%%%%%%%%%%%%%%%%%%%%%%%%%%%%%%%%%%%%%%%%%%%%%%%%%%%%%%%%%%%%%%%%%%%%%%%%%%%%%%%%%%%%%%%%%%%

\section{The $0$-PIM of a solvable restricted Lie algebra}

%%%%%%%%%%%%%%%%%%%%%%%%%%%%%%%%%%%%%%%%%%%%%%%%%%%%%%%%%%%%%%%%%%%%%%%%%%%%%%%%%%%%%%%%%%%%%%%%%%%

Let $L$ be a finite-dimensional restricted Lie algebra over a field $\F$ of prime
characteristic. It follows from Proposition \ref{upbd} that
\begin{equation}
\label{eq:max0PIM}
\dim_\F P_L(\F)\le p^{\dim_\F L-\mt(L)}\,.
\end{equation}
We say that $L$ has {\em maximal $0$-PIM\/} if $\dim_\F P_L(\F)=p^{\dim_\F L-
\mt(L)}$. In this case it follows from Corollary \ref{proj} that $P_L(\F)\cong\ind_T^L
(\F,0)$ for any torus $T$ in $L$ of maximal dimension. Our goal in this section is to
prove that any finite-dimensional solvable restricted Lie algebra $L$ has maximal
$0$-PIM. In particular, the $0$-PIM is induced from the one-dimensional trivial module
of a torus of maximal dimension. The next result will be important in the induction
step of the proof of Theorem \ref{solvmax0PIM} and in the proof of Theorem
\ref{max0PIM} (see also \cite[Theorem VII.14.3]{HB} for the group-theoretic
analogue).

\begin{lem}\label{ext}
Let $L$ be a finite-dimensional restricted Lie algebra over a field of prime characteristic
$p$, and let $I $ be a $p$-ideal of $L$. Then $L$ has maximal $0$-PIM if, and only if,
$I$ and $L/I$ have maximal $0$-PIM.
\end{lem}

\begin{proof}
Suppose first that $L$ has maximal $0$-PIM. Then it follows from Theorem \ref{fong}
and \cite[Lemma 1.2.6(2)(a)]{StrI} that
\begin{equation}
\label{eq:dim}
[\dim_\F P_I(\F)]\cdot[\dim_\F P_{L/I}(\F)]=p^{\dim_\F I-\mt(I)}\cdot p^{\dim_\F L/I-\mt(L/I)}\,.
\end{equation}
Suppose that $\dim_\F P_I(\F)<p^{\dim_\F I-\mt(I)}$. Then $\dim_\F P_{L/I}(\F)>
p^{\dim_\F L/I-\mt(L/I)}$, which contradicts \eqref{eq:max0PIM}. Hence $\dim_\F
P_I(\F)=p^{\dim_\F I-\mt(I)}$, and therefore also $\dim_\F P_{L/I}(\F)=p^{\dim_\F
L/I-\mt(L/I)}$. Consequently, $I$ and $L/I$ have maximal $0$-PIM. Finally, the
``if"-part is immediate from the equality in \eqref{eq:dim}, Theorem \ref{fong}, and
\cite[Lemma~1.2.6(2)(a)]{StrI}.
\end{proof}

Now we are ready to prove the first main result of this paper.

\begin{thm}\label{solvmax0PIM}
Every finite-dimensional solvable restricted Lie algebra over a field of prime
characteristic has maximal $0$-PIM.
\end{thm}

\begin{proof}
We may proceed by induction on the dimension of the solvable restricted Lie algebra
$L$. If $L$ is abelian, then the assertion is a consequence of \cite[Satz II.3.2]{Fe0}
(see also \cite[Corollary 1]{Fe2}). In particular, it holds if $\dim_\F L=1$. Thus we may
assume that $L$ is not abelian, and that the claim holds for all Lie algebras of dimension
less than $\dim_\F L$.

As $L$ is solvable but not abelian, the $p$-subalgebra $I:=\langle[L,L]\rangle_p$ is
a non-zero proper $p$-ideal of $L$ (see \cite[Exercise 2.1.2 and Proposition 2.1.3(4)]{SF}).
By induction hypothesis, $I$ and $L/I$ have maximal $0$-PIM, and therefore it follows
from Lemma~\ref{ext} that $L$ also has maximal $0$-PIM.
\end{proof}

\noindent {\bf Remark.} It is easy to see directly that one-dimensional restricted Lie
algebras have maximal $0$-PIM. As in the proof of Theorem \ref{max0PIM} one could
assume in the proof of Theorem \ref{solvmax0PIM} that the ground field is algebraically
closed and then one could use in the induction step that finite-dimensional solvable
restricted Lie algebras over algebraically closed fields have a $p$-ideal of codimension
one.
\vspace{.3cm}

According to the conjugacy of maximal tori due to David John Winter, any maximal torus
of a finite-dimensional solvable restricted Lie algebra $L$ has dimension $\mt(L)$ (see
\cite[Proposition 2.17]{Win} or also \cite[Theorem 1.5.6]{StrI}). As a consequence of
this in conjunction with Theorem \ref{solvmax0PIM}, we obtain the following generalization
of \cite[Satz II.3.2]{Fe0}, \cite[Corollary 1]{Fe2}, \cite[Corollary 4.5]{Fa2},
\cite[Proposition~2.2]{Fa3}, and \cite[Proposition~1]{Fe4} to solvable restricted Lie
algebras:

\begin{cor}\label{0PIM}
Let $L$ be a finite-dimensional solvable restricted Lie algebra over a field $\F$ of prime
characteristic $p$, and let $\chi$ be a linear form on $L$. If $S$ is a one-dimensional
$L$-module with $p$-character $\chi$, then $P_L(S)\cong\ind_T^L(S_{\vert T},\chi)$
holds for every maximal torus $T$ of $L$.
\end{cor}

\begin{proof}
As a consequence of Winter's theorem on the conjugacy of maximal tori, any maximal
torus $T$ of $L$ is a torus of maximal dimension $\mt(L)$ (see \cite[Proposition~2.17]{Win}
or \cite[Theorem 1.5.6]{StrI}). Now it follows from \cite[Lemma 1]{Fe5} and Theorem~\ref{solvmax0PIM}
that $$\dim_\F P_L(S)=\dim_\F P_L(\F)=p^{\dim_\F L-\mt(L)}\,,$$ and the assertion is
an immediate consequence of Corollary \ref{proj}.
\end{proof}

%%%%%%%%%%%%%%%%%%%%%%%%%%%%%%%%%%%%%%%%%%%%%%%%%%%%%%%%%%%%%%%%%%%%%%%%%%%%%%%%%%%%%%%%%%%%%%%%%%%

\section{An upper bound for the number of irreducible modules}

%%%%%%%%%%%%%%%%%%%%%%%%%%%%%%%%%%%%%%%%%%%%%%%%%%%%%%%%%%%%%%%%%%%%%%%%%%%%%%%%%%%%%%%%%%%%%%%%%%%

As an application of Theorem \ref{solvmax0PIM} we obtain an upper bound for the number
of the isomorphism classes of irreducible modules with a fixed $p$-character for a solvable
restricted Lie algebra. This result generalizes \cite[Theorem 4]{Fe5} and at the same time
simplifies the proof considerably.

\begin{thm}\label{numirr}
Let $L$ be a finite-dimensional solvable restricted Lie algebra over an algebraically closed
field $\F$ of prime characteristic $p$, and let $\chi$ be a linear form on $L$. Then the number
of isomorphism classes of irreducible $L$-modules with $p$-character $\chi$ is at most $p^{\mt(L)}$.
\end{thm}

\begin{proof}
Let $\irr(L,\chi)$ denote the set of isomorphism classes of irreducible $L$-modules with
$p$-character $\chi$, and let $\vert\irr(L,\chi)\vert$ denote its cardinality. It follows
from Proposition \ref{reciproc} and Theorem \ref{solvmax0PIM} that
\begin{eqnarray*}
p^{\dim_\F L}=\dim_\F u(L,\chi) & = & \sum_{[S]\in\irr(L,\chi)}(\dim_\F S)[\dim_\F P_L(S)]\\
& = & \sum_{[S]\in\irr(L,\chi)}\dim_\F[P_L(S)\otimes S^*]\\
& \ge & \vert\irr(L,\chi)\vert\cdot\dim_\F P_L(\F)\\
& = & \vert\irr(L,\chi)\vert\cdot p^{\dim_\F L-\mt(L)}\,.
\end{eqnarray*}
Dividing both sides of the inequality by $p^{\dim_\F L-\mt(L)}$ yields $\vert\irr(L,\chi)
\vert\le p^{\mt(L)}$.
\end{proof}

\noindent {\bf Remark.} If one uses \cite[Theorem 2]{Fe2} instead of Proposition \ref{reciproc},
then the proof of Theorem \ref{numirr} yields a new proof of \cite[Satz~6]{Str1}.
\vspace{.3cm}

Of course, the proof of Theorem \ref{numirr} holds for any restricted Lie algebra with maximal
$0$-PIM, but we will show in the last section that, at least in characteristic $p>3$, such Lie
algebras are necessarily solvable.

%%%%%%%%%%%%%%%%%%%%%%%%%%%%%%%%%%%%%%%%%%%%%%%%%%%%%%%%%%%%%%%%%%%%%%%%%%%%%%%%%%%%%%%%%%%%%%%%%%%

\section{Induced representations of filtered restricted Lie algebras}

%%%%%%%%%%%%%%%%%%%%%%%%%%%%%%%%%%%%%%%%%%%%%%%%%%%%%%%%%%%%%%%%%%%%%%%%%%%%%%%%%%%%%%%%%%%%%%%%%%%

Let $L$ be a restricted Lie algebra over a field of prime characteristic $p$ with a descending
restricted filtration $$L=L_{(-1)}\supset L_{(0)}\supset L_{(1)}\supset\cdots\supset L_{(h)}
\supset L_{(h+1)}=0$$ of depth $1$ and height $h$. For integers $n>h$ put $L_{(n)}:=0$
and for integers $m<-1$ put $L_{(m)}:=L$. Recall that $[L_{(m)},L_{(n)}]\subseteq L_{(m+n)}$
for all $m,n\in\Z$ and $L_{(n)}^{[p]}\subseteq L_{(pn)}$ for every $n\in\Z$ (see \cite[Section
1.9 and Section 3.1]{SF}. In particular, $L_{(n)}$ is a $p$-subalgebra of $L$ for every $n\in\Z$
and $L_{(n)}$ is $p$-unipotent for every integer $n\ge1$.

The graded vector space $\gr(L):=\bigoplus_{n\in\Z}\gr_n(L)$, where $\gr_n(L):=L_{(n)}/
L_{(n+1)}$ for every $n\in\Z$, associated with the filtration $(L_{(n)})_{n\in\Z}$ carries
canonically the structure of a graded restricted Lie algebra with bracket $[\cdot,\cdot]\colon
\gr(L)\times\gr(L)\to\gr(L)$ given by $$[x+L_{(m+1)},y+L_{(n+1)}]=[x,y]+L_{(m+n+1)}
\quad\mbox{ for all }m,n\in\Z,x\in L_{(m)},y\in L_{(n)}$$ and $[p]$-mapping $(\cdot)^{[p]}
\colon\gr(L)\to\gr(L)$ given by $$(x+L_{(n+1)})^{[p]}=x^{[p]}+L_{(pn+1)}\quad\mbox{
for every }n\in\Z\mbox{ and for every }x\in L_{(n)}$$ (see \cite[Theorem 3.3.1]{SF}).

Any descending filtration $(L_{(n)})_{n\in\Z}$ of a Lie algebra $L$ induces a descending
filtration $(U(L)_{(n)})_{n\in\Z}$ on the universal enveloping algebra $U(L)$ of $L$, where
$$U(L)_{(n)}:=\sum_{\substack{s\geq 0,\ n_j\geq -1\\n_1+\cdots+n_s\geq n}}L_{(n_1)}
\cdots L_{(n_s)}\,.$$ Moreover, if $\gr(U(L)):=\bigoplus_{n\in\Z} \gr_n(U(L))$, where
$\gr_n(U(L)):=U(L)_{(n)}/U(L)_{(n+1)}$ for every $n\in\Z$, denotes the associated
graded algebra, one has a canonical isomorphism $\phi_\bullet\colon U(\gr(L))\to\gr(U(L))$
induced by $\phi(x+L_{(n+1)}):=x+U(L)_{(n+1)}$ for every $n\in\Z$ and every $x
\in L_{(n)}$ (cf.\ \cite[Theorem 1.9.5]{SF} for ascending filtrations).

Let $u(L):=u(L,0)$ denote the restricted universal enveloping algebra of a restricted
Lie algebra $L$. In particular, one has a surjective homomorphism of algebras $\pi\colon
U(L)\to u(L)$. Then $u(L)$ carries the filtration $(u(L)_{(n)})_{n\in\Z}$ given by $u(L)_{(n)}
:=\pi(U(L)_{(n)})$ for every $n\in\Z$. Thus, if $\gr(u(L)):=\bigoplus_{n\in\Z}\gr_n
(u(L))$, where $\gr_n(u(L)):=u(L)_{(n)}/u(L)_{(n+1)}$ for every $n\in\Z$, denotes
the graded algebra associated with the filtration $(u(L)_{(n)})_{n\in\Z}$, by construction,
one has a canonical surjective homomorphism of graded algebras $\gr(\pi)\colon\gr(U(L))
\to\gr(u(L))$.

Suppose now that $(L_n)_{n\in\Z}$ is a restricted filtration of the restricted Lie algebra
$L$ as above. Then $\gr(L)$ is a graded restricted Lie algebra, and one has a surjective
homomorphism of graded algebras $\pi_{\gr}\colon U(\gr(L))\to u(\gr(L))$ with respect
to the total degree. By construction, one has a homomorphism of graded restricted Lie
algebras $\varphi\colon\gr(L)\to\gr(u(L))$ defined by $\varphi(x+L_{(n+1)}):=x+u
(L)_{(n+1)}$, satisfying $\varphi((x+L_{(n+1)})^{[p]})=\varphi(x+L_{(n+1)})^p$ for
every $n\in\Z$ and every $x\in L_{(n)}$. Hence $\varphi$ induces a homomorphism of
graded algebras $\varphi_\bullet\colon u(\gr(L))\to\gr(u(L))$, i.e., one has a commutative
diagram
\begin{equation*}
\xymatrix{
U(\gr(L))\ar[r]^{\phi_\bullet}\ar[d]_{\pi_{\gr}}&\gr(U(L))\ar[d]^{\gr(\pi)}\\
u(\gr(L))\ar[r]^{\varphi_\bullet}&\gr(u(L))
}
\end{equation*}
In particular, $\varphi_\bullet$ is surjective. The following result will be important for our purpose.

\begin{pro}\label{gr}
If $(L_{(n)})_{n\in\Z}$ is a descending restricted filtration of a restricted Lie algebra $L$ of
depth $1$ and height $h$, then $\varphi_\bullet\colon u(\gr(L))\to\gr(u(L))$ is an isomorphism
of graded algebras.
\end{pro}

\begin{proof}
Let $X=\bigsqcup_{-1\leq n\leq h} X_n$ be an ordered basis of $L$ such that
\begin{itemize}
\item[(i)] for $x\in X_m$ and $y\in X_n$ with $m<n$ one has $x<y$,
\item[(ii)] $\{x+L_{(n+1)}\mid x\in X_n\,\}$ is a basis of $\gr_n(L)$.
\end{itemize}
Let $\fF$ denote the set of functions from $X$ to $\{0,\dots,p-1\}$ with finite support. For
$\alpha\in\fF$ let $x^{\alpha}\in u(L)$ denote the monomial $\prod_{x\in X}x^{\alpha(x)}$,
where the product is taken with respect to increasing order. For $\alpha\in\fF$ put $$\vert
\alpha\vert:=\sum_{-1\leq n\leq h}n\cdot\sum_{x\in X_n}\alpha(x)\in\Z\,.$$ Then, by
\cite[Proposition~1.9.1]{SF} and Jacobson's analogue of the Poincar\'e-Birkhoff-Witt
theorem (see \cite[Theorem 2.5.1(2)]{SF}) $\{x^\alpha+u(L)_{(n+1)}\mid\vert\alpha\vert=n\}$
is a basis of $\gr_n(u(L))$, and $\{\prod_{-1\leq n\leq h}\prod_{x\in X_n}(x+L_{(n+1)})^{\alpha(x)}
\mid\vert\alpha\vert=n\}$, where products are taken with respect to increasing order in $X$,
is a basis of the homogeneous component of $u(\gr(L))$ of total degree $n$. Hence, since
$\varphi$ is mapping a basis injectively onto a basis, $\varphi_\bullet$ is an isomorphism.
\end{proof}

From the proof of Proposition \ref{gr} one also deduces the following properties:

\begin{cor}\label{aug}
If $(L_{(n)})_{n\in\Z}$ is a descending restricted filtration of a restricted Lie algebra $L$ of
depth $1$ and height $h$, then the following statements hold:
\begin{itemize}
\item[(1)] $u(L)_{(1)}\subseteq u(L)\cdot\fu(L_{(1)})$, where $\fu(L_{(1)}):=\Ker(u(L_{(1)})
                 \to\F)$ denotes the augmentation ideal of $u(L_{(1)})$.
\item[(2)] $u(L)_{(0)}\subseteq u(L_{(0)})+u(L)\cdot\fu(L_{(1)})$.
\end{itemize}
\end{cor}

Let $(L_{(n)})_{n\in\Z}$ be a descending restricted filtration of a restricted Lie algebra $L$ of
depth $1$ and height $h$. Then every restricted $\gr_0(L)$-module $V$ can be considered as
a restricted $L_{(0)}$-module. The induced module $$M:=\ind_{L_{(0)}}^L(V,0):=u(L)\otimes_{u
(L_{(0)})}V$$ has a descending filtration $(M_{(r)})_{r\in\Z}$ given by $M_{(r)}:=u(L)_{(r)}
(1\otimes V)$, i.e., for all $r,s\in\Z$ one has
\begin{equation}
\label{eq:filtr}
u(L)_{(r)}\cdot M_{(s)}\subseteq M_{(r+s)}\,.
\end{equation}
Since the restricted $L_{(0)}$-module $V$ is inflated from $\gr_0(L)$, Corollary \ref{aug} implies
that $M_{(0)}=1\otimes V$, and $M_{(r)}=0$ for every integer $r>0$.

Let $\gr(M):=\bigoplus_{r\in\Z}\gr_r(M)$, where $\gr_r(M):=M_{(r)}/M_{(r+1)}$ for every
$r\in\Z$, denote the associated graded vector space. According to \eqref{eq:filtr}, $\gr(M)$
is a $\gr(u(L))$-module. Thus, by Proposition \ref{gr}, $\gr(M)$ may be considered as a restricted
$\gr(L)$-module.

\begin{thm}\label{grind}
Let $L$ be a restricted Lie algebra, and let $(L_{(n)})_{n\in\Z}$ be a descending restricted
filtration of $L$ of depth $1$ and height $h$. Let $V$ be a restricted $\gr_0(L)$-module, and
let $M:=\ind_{L_{(0)}}^L(V,0)$. Then one has a canonical isomorphism of graded restricted
$\gr(L)$-modules $$\ind_{\gr_+(L)}^{\gr(L)}(V,0)\cong\gr(M)\,,$$ where $\gr_+(L):=
\bigoplus_{n\geq 0}\gr_n(L)$. In particular, if $\ind_{\gr_+(L)}^{\gr(L)}(V,0)$ is an irreducible
$\gr(L)$-module, then $M$ is an irreducible $L$-module.
\end{thm}

\begin{proof}
It follows from $M_{(0)}=1\otimes V$and $M_{(r)}=0$ for every integer $r>0$ that one has
an isomorphism $\iota\colon V\to\gr_0(M)$ of graded $u(\gr_+(L))$-modules. As induction is
the left-adjoint of the restriction functor, this yields a mapping of graded $u(\gr(L))$-modules $\iota_\bullet\colon\ind_{\gr_+(L)}^{\gr(L)}(V,0)\to\gr(M)$.
By construction, $\iota_\bullet$ is surjective.

Let $X_{-1}\subseteq L$ be an ordered set such that $\{x+L_{(0)}\mid x\in X_{-1}\}$ is a
basis of $\gr_{-1}(L)$. Then, using the same notation as used in the proof of Proposition \ref{gr},
one obtains for $r<0$ an isomorphism of vector spaces $$\gr_r(M)\cong\bigoplus_{\substack{
\alpha\in\fF,\vert\alpha\vert=r}}\F x^\alpha\otimes V\,,$$ where $\vert\alpha\vert:=-\sum_{x
\in X_{-1}}\alpha(x)$. By Jacobson's analogue of the Poincar\'e-Birkhoff-Witt theorem (see
\cite[Theorem 2.5.1(2)]{SF}), one has an isomorphism of graded vector spaces $$u(\gr(L))
\otimes_{u(\gr_+(L))}\F\cong\F[X_{-1}]/\F[X_{-1}]\{x^p\mid x\in X_{-1}\}\,.$$ This yields
the desired isomorphism. The final remark follows from the fact that any non-zero proper
$L$-submodule $S$ of $M$ gives rise to a non-zero proper $\gr(L)$-module $\gr(S)$ of
$\gr(M)$, where $\gr_r(S):=(S\cap M_{(r)})+M_{(r+1)}/M_{(r+1)}$ for every $r\in\Z$.
\end{proof}

%%%%%%%%%%%%%%%%%%%%%%%%%%%%%%%%%%%%%%%%%%%%%%%%%%%%%%%%%%%%%%%%%%%%%%%%%%%%%%%%%%%%%%%%%%%%%%%%%%%

\section{Representations of non-graded Hamiltonian Lie algebras}

%%%%%%%%%%%%%%%%%%%%%%%%%%%%%%%%%%%%%%%%%%%%%%%%%%%%%%%%%%%%%%%%%%%%%%%%%%%%%%%%%%%%%%%%%%%%%%%%%%%

Let $\caO(2;\uone)=\F[x,y]/\F[x,y]\{ x^p,y^p\}$ denote the truncated polynomial algebra
in two variables $x$ and $y$ over a field $\F$ of prime characteristic $p$. By $\caO(2;\uone)^{(1)}$
we denote its augmentation ideal, i.e., $f\in \caO(2;\uone)^{(1)}$ if, and only if, $f$ has zero
constant term. Moreover, $\der_x=\der/\der x$ and $\der_y=\der/\der y$ will denote the partial
derivatives of the algebra $\caO(2;\uone)$.

The Hamiltonian Lie algebra $H(2;\uone)$ is defined as the subalgebra of the Jacobson-Witt algebra
$W(2;\uone)$ whose elements annihilate the standard $2$-form $\omega_H:=\dx\wedge\dy$, i.e.,
one has
\begin{align}
H(2;\uone)&:=\{D\in W(2;\uone)\mid D(\omega_H)=0\}\label{eq:h21}\\
&\,\,=\{f\der_x+g\der_y\in W(2;\uone)\mid\der_y(g)=-\der_x(f)\}\,.\label{eq:h22}
\end{align}
Since $W(2;\uone)$ is a graded restricted Lie algebra, \eqref{eq:h21} implies that $H(2;\uone)$ is
also a graded restricted Lie algebra. Let $\{\cdot,\cdot\}\colon\caO(2;\uone)\times\caO(2;\uone)
\to\caO(2,\uone)$ denote the standard Poisson bracket given by
\begin{equation}
\label{eq:poissonstand}
\{f,g\}=(\der_xf)(\der_yg)-(\der_yf)(\der_xg)\,\,\mbox{ for all }f,g\in\caO(2;\uone)\,.
\end{equation}
Then $(\caO(2;\uone),\{\cdot,\cdot\})$ is a Lie algebra, and the canonical map (of degree $-2$)
\begin{equation}
\label{eq:caD}
\boD\colon(\caO(2;\uone),\{\cdot,\cdot\})\to W(2;\uone)\,,\ \ 
\boD(f):=\der_x(f)\der_y-\der_y(f)\der_x\,\,\mbox{ for }f\in\caO(2;\uone)\,,
\end{equation}
defines a Lie algebra homomorphism satisfying $\Ker(\boD)=\F 1$ and $\im(\boD)\subseteq
H(2;\uone)$. It is well known that $\im(\boD)$ is an ideal in $H(2;\uone)$ containing the derived
subalgebra of $H(2;\uone)$. Moreover, for $p\geq3$
\begin{equation}\label{eq:defH2sim}
H(2;\uone)^{(2)}=\langle\{\boD(x^ay^b)\mid 0<a+b<2(p-1)\}\rangle_\F\subset H(2;\uone)
\end{equation}
is a simple ideal of $H(2;\uone)$. For further details see \cite[\S 4.2]{StrI} and \cite[\S 4.4]{SF}.
(Note that in \cite{SF}, the algebra $H(2;\uone)$ is denoted by $H^{\prime\prime}(2;\duone)$,
and that $\im(\boD)$ is denoted by $H^\prime(2;\duone)$.)

Define $\Gamma:=x^{p-1}\der_y$ and $\Theta:=-y^{p-1}\der_x$. Then, by \eqref{eq:h22},
$\Gamma,\Theta\in H(2;\uone)$. Since $x^{p-1}\not\in\im(\der_x)$ and $y^{p-1}\not\in\im
(\der_y)$, \eqref{eq:caD} implies that $\Gamma,\Theta\not\in\im(\boD)$. Hence
\begin{equation}\label{eq:algY}
\boY:=\im(\boD)\oplus\F\Gamma\oplus\F\Theta\subseteq H(2;\uone)
\end{equation}
is a graded subalgebra of $H(2;\uone)$ of dimension $p^2+1$. (Indeed, it follows from
\cite[Proposition 2.1.8(b)(ii)]{BW} that $\dim_\F H(2;\uone)=p^2+1$, so $\boY=H(2;\uone)$,
but this is irrelevant for our purpose.) From the identities
\begin{align}
[\Gamma,\boD(x^ay^b)]&=b\,\boD(x^{p-1+a}y^{b-1})\,,\label{eq:idH1}\\
[\Theta,\boD(x^ay^b)]&=-a\,\boD(x^{a-1}y^{p-1+b})\,,\label{eq:idH2}\\
[\Gamma,\Theta]&=\boD(x^{p-1}y^{p-1})\,,\label{eq:idH3}\\
[\boD(x),\boD(x^ay^b)]&=b\,\boD(x^ay^{b-1})\,,\label{eq:idH4}\\
[\boD(y),\boD(x^ay^b)]&=-a\,\boD(x^{a-1}y^{b})\,,\label{eq:idH5}\\
\boD(x^ay^b)^{[p]}&=
\begin{cases}
\hfil 0\hfil&\ \ \text{if $(a,b)\not=(1,1)$\,,}\\
\hfil \boD(xy)\hfil&\ \ \text{if $a=b=1$\,,}
\end{cases}
\label{eq:idH6}\\
\Gamma^{[p]}&=0\,,\label{eq:idH7}\\
\Theta^{[p]}&=0\,.\label{eq:idH8}
\end{align}
for all $a,b\in\{0,\dots,p-1\}$ such that $(a,b)\not=(0,0)$, one concludes that $\boY$ is a
graded $p$-subalgebra of $H(2;\uone)$. Moreover, by construction, one has the
following property:

\begin{pro}\label{coker}
The graded vector space $\coker(H(2;\uone)^{(2)}\to\boY)$ is concentrated in degrees $p-2$
and $2p-4$.
\end{pro}

\noindent {\bf Remark.}  Note that $\boY_0:=\boY\cap H(2;\uone)_0\cong\fsl_2(\F)$ (see
\cite[Proposition 4.4.4(4)]{SF}).
\vspace{.3cm}

The Hamiltonian Lie algebra $H(2;\uone;\Phi(\tau))$ is defined as the subalgebra of the Jacobson-Witt
algebra $W(2;\uone)$ whose elements annihilate the $2$-form $$\omega_{\Phi(\tau)}:=(1+x^{p-1}
y^{p-1})\,\dx\wedge\dy\,,$$ i.e., one has
\begin{equation}\label{eq:h2tau}
H(2;\uone;\Phi(\tau)):=\{D\in W(2;\uone)\mid D(\omega_{\Phi(\tau)})=0\}\,.
\end{equation}
In particular, $H(2;\uone;\Phi(\tau))$ is a $p$-subalgebra of $W(2;\uone)$, and an element $f\der_x
+g\der_x\in W(2;\uone)$ is contained in $H(2;\uone;\Phi(\tau))$ if, and only if,
\begin{equation}\label{eq:defeqH2tau}
(1+x^{p-1}y^{p-1})(\der_xf+\der_yg)-x^{p-2}y^{p-2}(yf+xg)=0\,.
\end{equation}
For the distinguished element $\Lambda=1-x^{p-1}y^{p-1}\in\caO(2;\uone)$ one defines the Poisson
bracket $\{\cdot,\cdot\}_\Lambda\colon\caO(2;\uone)\times\caO(2;\uone)\to\caO(2;\uone)$ by
\begin{equation}\label{eq:poisson}
\{f,g\}_\Lambda=\big((\der_xf)(\der_yg)-(\der_yf)(\der_xg)\big)\Lambda\,\,\mbox{ for all }f,g\in
\caO(2,\uone)\,.
\end{equation}
Then $(\caO(2;\uone),\{\cdot,\cdot\}_\Lambda)$ is a Lie algebra, and the map $\boD_\Lambda=\Lambda
\cdot\boD\colon\caO(2;\uone)\to W(2;\uone)$, where $\boD$ is defined as in \eqref{eq:caD}, is a
homomorphism of Lie algebras satisfying $\Ker(\boD_\Lambda)=\F 1$ and $H:=\im(\boD_\Lambda)
\subseteq H(2;\uone;\Phi(\tau))$. Moreover, one has
\begin{equation}\label{eq:monomD}
\{(a,b)\mid 0\leq a,b\leq p-1,\ \boD_\Lambda(x^ay^b)\not=\boD(x^ay^b)\,\}=\{(1,0),(0,1)\}\,.
\end{equation}
If $p>3$, then $H$ coincides with the derived subalgebra $H(2;\uone;\Phi(\tau))^{(1)}$ of $H(2;
\uone;\Phi(\tau))$ which is a simple Lie algebra. For further details the reader may wish to consult
\cite[\S 1]{Str2} and the references therein, as well as \cite[\S6.3]{StrI}.

Let $\Gamma:= x^{p-1}\der_y$, $\Theta:=-y^{p-1}\der_x\in W(2;\uone)$ be given as before.
Then, by \eqref{eq:defeqH2tau}, $\Gamma,\Theta\in H(2;\uone;\Phi(\tau))$, but $\Gamma,\Theta
\not\in H$, and
\begin{equation}\label{eq:defL}
L:=H\oplus\F\Gamma\oplus\F\Theta\,,
\end{equation}
is a $p$-subalgebra of $H(2;\uone;\Phi(\tau))$ that is isomorphic to $\Der(H)$ (cf.\ \cite[Proposition~1.1(1)]{Str2}).
Hence, for $p>3$, it follows from \cite[Theorem 10.3.1]{StrII} and the argument in the middle of
\cite[p.~41]{StrII} that $L$ coincides with the minimal $p$-envelope of $H$ in $W(2;\uone)$. The
following identities hold in $L$ (see \cite[p.~582]{Str2}):
\begin{align}
[\Gamma,\boD_\Lambda(x^ay^b)]&=b\,\boD_\Lambda(x^{p-1+a}y^{b-1})\,,\label{eq:idHtau1}\\
[\Theta, \boD_\Lambda(x^ay^b)]&=-a\,\boD_\Lambda(x^{a-1}y^{p-1+b})\,,\label{eq:idHtau2}\\
[\Gamma,\Theta]&=\boD_\Lambda(x^{p-1}y^{p-1})\,,\label{eq:idHtau3}\\
[\boD_\Lambda(x),\boD_\Lambda(x^ay^b)]&=
\begin{cases}
\hfil b\,\boD_\Lambda(x^ay^{b-1})\hfil&\ \ \text{if $(a,b)\not=(0,1)$\,,}\\
\hfil-\boD_\Lambda(x^{p-1}y^{p-1})\hfil&\ \ \text{if $(a,b)=(0,1)$\,,}
\end{cases}
\label{eq:idHtau4}\\
[\boD_\Lambda(y),\boD_\Lambda(x^ay^b)]&=
\begin{cases}
\hfil -a\,\boD_\Lambda(x^{a-1}y^b)\hfil&\ \ \text{if $(a,b)\not=(1,0)$\,,}\\
\hfil\boD_\Lambda(x^{p-1}y^{p-1})\hfil&\ \ \text{if $(a,b)=(1,0)$\,,}
\end{cases}
\label{eq:idHtau5}\\
\boD_{\Lambda}(x^ay^b)^{[p]}&=
\begin{cases}
\hfil 0\hfil&\ \ \text{if $(a,b)\not\in\{(0,1),(1,0),(1,1)\}$\,,}\\
\hfil \boD_\Lambda(xy)\hfil&\ \ \text{if $a=b=1$\,,}\\
\hfil\Gamma\hfil&\ \ \text{if $(a,b)=(1,0)$\,,}\\
\hfil\Theta\hfil&\ \ \text{if $(a,b)=(0,1)$\,,}
\end{cases}
\label{eq:idHtau6}\\
\Gamma^{[p]}&=0,\label{eq:idHtau7}\\
\Theta^{[p]}&=0\label{eq:idHtau8}\,.
\end{align}
The natural grading on $W(2;\uone)$ induces filtrations
\begin{equation}
\label{eq:filt}
H_{(n)}:=H\cap W(2;\uone)_{(n)},\qquad L_{(n)}:=L\cap W(2;\uone)_{(n)}\,\,\mbox{ for every }n\in\Z\,,
\end{equation}
where $W(2;\uone)_{(n)}=\bigoplus_{d\geq n}W(2;\uone)_d$. 
The following result will be important for our purpose.

\begin{pro}\label{filt}
The family $(L_{(n)})_{n\in\Z}$ is a descending restricted filtration of $L$ of depth $1$ and height $2p-4$.
Moreover, $\gr(L)\cong\boY$.
\end{pro}

\begin{proof}
According to \cite[Proposition~1.2(1)]{Str2}, one has that $L=L_{(-1)}$, $H=H_{(-1)}$, $H_{(n)}=L_{(n)}
=0$ for $n>2p-4$, and for $-1\leq n\leq 2p-4$:
\begin{equation}\label{eq:filH}
\begin{aligned}
H_{(n)}&=\F\boD_\Lambda(x^{p-1}y^{p-1})+\langle\{\boD_\Lambda(x^ay^b)\mid n+2\leq a+b
\leq 2p-3\}\rangle_\F\,,\\
L_{(n)}&=
\begin{cases}
H_{(n)}\oplus\F\Gamma\oplus\F\Theta&\ \text{for $-1\leq n\leq p-2$\,,}\\
\hfil\hfil H_{(n)}\hfil&\ \text{for $p-2<n\leq 2p-4$\,.}
\end{cases}
\end{aligned}
\end{equation}
From the identities \eqref{eq:idHtau6}--\eqref{eq:idHtau8} one obtains that $L_{(n)}^{[p]}\subseteq
L_{(pn)}$ for every integer $n\geq 0$. In the associated graded algebra $\gr(L)$ one has the additional
identity (see \eqref{eq:idHtau5})
\begin{equation}\label{eq:Yid}
[\boD_\Lambda(x)+L_{(0)},\boD_\Lambda(y)+L_{(0)}]=0\,.
\end{equation}
Hence using \eqref{eq:Yid} and the identities \eqref{eq:idHtau1}--\eqref{eq:idHtau5} one concludes
that the linear map $\psi\colon\gr(L)\to\boY$ given by
\begin{equation}\label{eq:defpsi}
\begin{aligned}
\psi(\boD_\Lambda(x^ay^b)+L_{(a+b-1)})&:=\boD(x^ay^b)\,,\\ 
\psi(\Gamma+L_{(p-1)})&:=\Gamma\,,\\
\psi(\Theta+L_{(p-1)})&:=\Theta\,,
\end{aligned}
\end{equation}
is an isomorphism of graded Lie algebras. The identities \eqref{eq:idHtau6}--\eqref{eq:idHtau8} imply
that the mapping $\psi\vert_{\gr_+(L)}\colon\gr_+(L)\to\boY$ is a homomorphism of restricted Lie
algebras. Thus, as $\boY$ is a graded restricted Lie algebra, the filtration $(L_{(n)})_{n\in\Z}$ is
restricted.
\end{proof}

It is well known that $L_{(0)}/L_{(1)}\cong\boY_0\cong\fsl_2(\F)$ (see \cite[Proposition 1.2(2.b)]{Str2}
and the remark after Proposition~\ref{coker}). For $\lambda\in\{0,\dots,p-1\}$ let $V(\lambda)$
denote the irreducible restricted $\fsl_2(\F)$-module of highest weight $\lambda$, i.e., $\dim_\F
V(\lambda)=\lambda+1$. By $V(\lambda)$ we will also denote its inflation to $L_{(0)}$. The following
result will be important for the proof of the main result in the next section.

\begin{thm}\label{irr}
Let $\F$ be an algebraically closed field of characteristic $p>3$. Then the restricted
$L$-module $M(\lambda):=\ind_{L_{(0)}}^L(V(\lambda),0)$ is irreducible for every
$\lambda\in\{2,\dots,p-1\}$.
\end{thm}

\begin{proof}
By virtue of Proposition \ref{filt}, $(L_{(n)})_{n\in\Z}$ is a restricted filtration of $L$ and
$\gr(L)\cong\boY$. Put $\boX:=H(2;\uone)^{(2)}\subseteq\boY$, and set $\boX_{(n)}:=
\boX\cap\boY_{(n)}$ for any $n\in\Z$. According to  Proposition \ref{coker}, $\coker
(\boX\to\boY)$ is concentrated in degrees $p-2$ and $2p-4$. Hence the canonical map
\begin{equation}\label{eq:Reps1}
\ind_{\boX_{(0)}}^{\boX}(V(\lambda)_{\vert\boX_{(0)}},0)\longrightarrow
\ind_{\boY_{(0)}}^{\boY}(V(\lambda),0)_{\vert\boX}
\end{equation}
is an isomorphism. As $\ind_{\boX_{(0)}}^{\boX}(V(\lambda)_{\vert\boX_{(0)}},0)$ is
an irreducible restricted $\boX$-module for $\lambda\neq 0,1$ (see \cite[p.\ 253]{Hol}),
this implies that $\ind_{\boY_{(0)}}^{\boY}(V(\lambda),0)$ is an irreducible $\boY$-module.
Consequently, by Theorem \ref{grind}, $M(\lambda)$ is an irreducible $L$-module unless
$\lambda=0,1$.
\end{proof}

\noindent {\bf Remark.} Since the induced modules in Theorem \ref{irr} coincide with
the corresponding coinduced modules, and as Shen's mixed products are isomorphic to
coinduced modules (see \cite[Corollary 2.6]{Fa1} and \cite[Proposition 1.5(1)]{FS}),
the irreducibility of $\ind_{\boX_{(0)}}^{\boX}(V(\lambda)_{\vert\boX_{(0)}},0)$ in
the proof of Theorem \ref{irr} follows also from \cite[Proposition~1.2]{S}.

%%%%%%%%%%%%%%%%%%%%%%%%%%%%%%%%%%%%%%%%%%%%%%%%%%%%%%%%%%%%%%%%%%%%%%%%%%%%%%%%%%%%%%%%%%%%%%%%%%%

\section{Restricted Lie algebras with maximal $0$-PIM}

%%%%%%%%%%%%%%%%%%%%%%%%%%%%%%%%%%%%%%%%%%%%%%%%%%%%%%%%%%%%%%%%%%%%%%%%%%%%%%%%%%%%%%%%%%%%%%%%%%%

Our goal in this section is to prove that for fields of characteristic $p>3$ finite-dimensional
restricted Lie algebras having maximal $0$-PIM are necessarily solvable. The following
criterion for a restricted Lie algebra having maximal $0$-PIM will be very useful in the
proof of this result. For a Lie algebra $L$ and any $L$-module $M$ we denote by $M^L
:=\{m\in M\mid\forall\,x\in L:x\cdot m=0\}$ the {\it space of $L$-invariants\/} of $M$.

\begin{lem}\label{inv}
Let $L$ be a finite-dimensional restricted Lie algebra over a field of prime
characteristic, and let $T $ be a torus in $L$ of maximal dimension. Then $L$
has maximal $0$-PIM if, and only if, $S^T=0$ for every non-trivial irreducible
restricted $L$-module $S$.
\end{lem}

\begin{proof}
Suppose first that $L$ has maximal $0$-PIM. Then it follows from Corollary~\ref{proj}
that $P_L(\F)\cong\ind_T^L(\F,0)$.  Since $\Hom_L(P_L(\F),S)=0$ for every
non-trivial irreducible restricted $L$-module $S$, Frobenius reciprocity yields:
\begin{equation}
\label{eq:inv}
S^T\cong\Hom_T(\F,S_{\vert T})\cong\Hom_L(\ind_T^L
(\F,0),S)\cong\Hom_L(P_L(\F),S)=0
\end{equation}
for every non-trivial irreducible restricted $L$-module $S$.

Conversely, suppose that $S^T=0$ for every non-trivial irreducible restricted
$L$-module $S$. Reading \eqref{eq:inv} backwards shows that the head of
$\ind_T^L(\F,0)$ is a trivial $L$-module. On the other hand, applying Frobenius
reciprocity to the one-dimensional trivial irreducible $L$-module, one has $\Hom_L
(\ind_T^L(\F,0),\F)\cong\Hom_T(\F,\F)\cong\F$. This implies that the head of
$\ind_T^L(\F,0)$ is one-dimensional, and therefore $\ind_T^L(\F,0)$ is
indecomposable. By virtue of Proposition \ref{upbd}, $P_L(\F)$ is a direct
summand of $\ind_T^L(\F,0)$. As a consequence one obtains that $P_L(\F)
\cong\ind_T^L(\F,0)$. Hence $L$ has maximal $0$-PIM.
\end{proof}

As we could not find a reference for the following result, which will be needed in
the proof of Theorem \ref{max0PIM}, it is included here for completeness.

\begin{lem}\label{psim}
Let $L$ be a non-abelian restricted Lie algebra over a field of prime characteristic
$p$ with no non-zero proper $p$-ideals. Then the following statements hold:
\begin{enumerate}
\item[(1)] $L$ is a minimal $p$-envelope of $[L,L]$.
\item[(2)] $[L,L]$ is simple as an ordinary Lie algebra.
\item[(3)] $[L,L]$ is a non-trivial irreducible $L$-module.
\end{enumerate}
\end{lem}

\begin{proof}
(1): Consider the $p$-ideal $\langle[L,L]\rangle_p$ of $L$ (cf.\ \cite[Proposition 2.1.3(4)]{SF}).
Then $0\neq [L,L]\subseteq\langle[L,L]\rangle_p$, and thus by assumption $\langle[L,L]
\rangle_p=L$. Accordingly, $L$ is a $p$-envelope of $[L,L]$. As the center $C(L)$ of $L$ is a
$p$-ideal of $L$ and $L$ is not abelian, it follows by assumption that $C(L)=0\subseteq
[L,L]$. Hence \cite[Theorem 2.5.8(3)]{SF} implies that $L$ is a minimal $p$-envelope of
$[L,L]$.

(2): Suppose that $I$ is a non-zero ideal of $[L,L]$. By virtue of (1) and \cite[Proposition
2.1.3(1)]{SF}, we have $$[L,I]=[\langle[L,L]\rangle_p,I]=[\sum_{n\in\Z_{\ge 0}}\langle
[L,L]^{[p]^n}\rangle_\F,I]=\sum_{n\in\Z_{\ge 0}}\ad([L,L])^{p^n}(I)\subseteq I\,,$$
i.e., $I$ is an ideal of $L$. Then by \cite[Proposition 2.1.3(4)]{SF} $\langle I\rangle_p$
is a non-zero $p$-ideal of $L$, and therefore by assumption $\langle I\rangle_p=L$.
Applying \cite[Proposition~2.1.3(2)]{SF} yields $$[L,L]=[\langle I\rangle_p,\langle
I\rangle_p]=[I,I]\subseteq I\,.$$ Consequently, $I=[L,L]$, and thus $[L,L]$ is simple as
an ordinary Lie algebra.

(3): The proof of (2) also shows that $[L,L]$ is an irreducible $L$-module which is non-trivial
as $C(L)=0$.
\end{proof}

\begin{thm}\label{max0PIM}
Let $L$ be a finite-dimensional restricted Lie algebra over a field of prime characteristic
$p>3$. Then $L$ has maximal $0$-PIM if, and only if, $L$ is solvable.
\end{thm}

\begin{proof}
One implication is just Theorem \ref{solvmax0PIM}, so that only the converse has to be
proved.

Let $\overline{\F}$ denote an algebraic closure of the ground field $\F$ of $L$, and set
$\overline{L}:=\overline{\F}\otimes_\F L$. As the trivial irreducible $L$-module $\F$ is
absolutely irreducible, we have that $$P_{\overline{L}}(\overline{\F})\cong\overline{\F}
\otimes_\F P_L(\F)\,.$$ Hence we may assume that the ground field of $L$ is algebraically
closed.

Suppose that there exists a non-solvable restricted Lie algebra $L$ with maximal $0$-PIM.
Furthermore, we may assume that $L$ has minimal dimension, i.e., every proper $p$-subalgebra
of $L$ is either solvable or does not have maximal $0$-PIM.

Suppose now in addition that $L$ has a non-zero proper $p$-ideal $I$. According to Lemma
\ref{ext}, $I$ and $L/I$ have maximal $0$-PIM. By the minimality of $L$, they both must be
solvable, but then also $L$ would be solvable which is a contradiction. So for the remainder
of the proof we may assume that $L$ has no non-zero proper $p$-ideal. Hence it follows
from Lemma \ref{psim} that the derived subalgebra $[L,L]$ of $L$ is a simple Lie algebra and
that $L$ is a minimal $p$-envelope of $[L,L]$.

We first discuss the case that $[L,L]$ has absolute toral rank one. By Engel's theorem there
exists an element $x$ of $[L,L]$ such that $\ad x$ is not nilpotent on $[L,L]$. Consider the
Jordan-Chevalley-Seligman-Schue decomposition $x=x_s+x_n$ of $x$ in the restricted Lie
algebra $L$, where $x_s$ and $x_n$ denote the semisimple part and the $p$-nilpotent part
of $x$, respectively (see \cite[Theorem 2.3.5]{SF}). Then $T:=\langle x_s\rangle_p$
is a torus of $L$ of maximal dimension and $[L,L]^T\supseteq\F x\neq 0$. By virtue of
Lemma~\ref{psim}(3), $[L,L]$ is a non-trivial irreducible restricted $L$-module. Hence we
conclude from Lemma \ref{inv} that $L$ does not have maximal $0$-PIM, which is a contradiction.

Let $H$ be any $2$-section of $L$ with respect to a torus $T$ of $L$ of maximal dimension
(cf.\ \cite[Definition 1.3.9]{StrI}). As $T\subseteq H$, we have $\mt(L)=\mt(H)$. Moreover,
it follows from Proposition \ref{ind} that $P_L(\F)$ is a direct summand of $\ind_H^L(P_H(\F),
0)$ and therefore \cite[Proposition 5.6.2]{SF} yields: $$\dim_\F P_H(\F)\ge\frac{\dim_\F P_L
(\F)}{p^{\dim_\F L-\dim_\F H}}=p^{\dim_\F H-\mt(H)}\,.$$ Hence by Proposition \ref{upbd},
$H$ has maximal $0$-PIM, and the minimality of $L$ implies that either $L=H$ or $H$ is solvable.
In the latter case we conclude from \cite[Theorem~1.3.10]{StrI} that $L$ must also be solvable,
which is a contradiction. Consequently, $\mt(L)=\mt(H)=2$, and thus it follows from \cite[Lemma
1.2.6(1)]{StrI} that $[L,L]$ has absolute toral rank two.

We use \cite[Theorem, p.~2]{StrII} and proceed by a case-by-case analysis. Note that the
simple Lie algebras $A_2$, $B_2$, $G_2$, $W(2;\uone)$, $S(3;\uone)^{(1)}$, $H(4;\uone
)^{(1)}$, $K(3;\uone)$, and $\mathcal{M}(1,1)$ are restricted, so that $[L,L]=L$. Hence in
all these cases one has $[L,L]^T\neq 0$ for any torus $T$ of $L$, which contradicts Lemma
\ref{inv}.

Next, consider $[L,L]\cong W(1;2)$. Suppose that $[L,L]\cap T=0$ for some two-dimensional
torus $T$ of $L$. Then it follows from \cite[Proposition 4.2.2(2)]{SF} and \cite[Theorem
7.2.2(1)]{StrI} that $\dim_{\F}([L,L]+T)>\dim_{\F}L$, which is a contradiction. Hence $$0
\neq [L,L]\cap T\subseteq [L,L]^T$$ for every two-dimensional torus $T$ of $L$, which
again contradicts Lemma \ref{inv}.

If $[L,L]\cong H(2;(1,2))^{(2)}$ or $[L,L]\cong H(2;\uone;\Phi(1))$, then it follows from
\cite[Lemma 10.2.3]{StrII} and \cite[Theorem 10.4.6(1)]{StrII}, respectively, that $$\dim_\F
[L,L]^T\ge\dim_\F([L,L]\cap T)=1$$ for every two-dimensional torus $T$ of $L$, which once
more contradicts Lemma \ref{inv}.

Thus we must have $[L,L]\cong H(2;\uone;\Phi(\tau))^{(1)}$. Let $T$ be a two-dimensional
torus in $L$. Moreover, let $(L_{(n)})_{-1\le n\le 2p-4}$ denote the descending restricted
filtration of $L$ considered in Section 5. Then $\gr_0(L)\cong\fsl_2(\F)$ (see \cite[Proposition
1.2(2.b)]{Str2}, and by \cite[Proposition 1.2(2.a)]{Str2}, $\gr_{-1}(L)$ is isomorphic as
an $L_{(0)}$-module to the module $V(1)$ inflated from the two-dimensional irreducible
$\fsl_2(\F)$-module $V(1)$. In particular, since every non-zero torus $T_0$ of $L_{(0)}$ is
one-dimensional and $0$ is not a weight of $T_0$ on $\gr_{-1}(L)\cong V(1)$, it follows
that $L_{(0)}\cap T=\{0\}$. Consequently, one concludes by comparison of dimensions
that $L=L_{(0)}\oplus T$. According to Theorem~\ref{irr}, the $L$-module $M(2):=\ind_{
L_{(0)}}^L(V(2),0)$ is irreducible. Now, by \cite[Theorem 2.3.6(1)]{SF}, there exists a
toral basis $\{t_1,t_2\}$ for $T$. As $L=L_{(0)}\oplus T$, the subspace $(1-t_1^{p-1})
(1-t_2^{p-1})M(2)$ of $M(2)$ is non-zero, and it is clear that it is annihilated by $T$. This
assures the existence of a non-trivial irreducible $L$-module $S:=M(2)$ such that $S^T
\neq 0$ which again contradicts Lemma \ref{inv}.
\end{proof}

\noindent {\bf Remark.} Note that it follows from \cite[Theorem 10.3.2(2)]{StrII} that $$[H
(2;\uone;\Phi(\tau))^{(1)}]^T=0$$ for every torus $T$ of maximal dimension in a minimal
$p$-envelope of $H(2;\uone;\Phi(\tau))^{(1)}$. This is the reason that the argument in
this case is much more involved than in the other cases.

In the case that $[L,L]\cong W(1;2)$ one can prove directly that $L$ does not have maximal
$0$-PIM. Namely, as $p>3$, it follows from \cite[Theorem 2.5.9]{N} in conjunction with
\cite[Theorem 7.2.2(1), Lemma 1.2.6(1), and Theorem 7.6.3(2)]{StrI} that $$\dim_{\F}P_L
(\F)=2p^{p^2-2}<p^{p^2-1}=p^{\dim_\F(L)-\mt(L)}\,.$$ This argument also motivated the
example at the end of this section.
\vspace{.3cm}

We conclude the paper by presenting an example of a non-solvable restricted Lie algebra in
characteristic two that has maximal $0$-PIM. In fact, the restricted Lie algebra in the example
is semisimple as an ordinary Lie algebra (see \cite[Proposition~1.4(2)]{Fe6}) and simple as a
restricted Lie algebra (see \cite[Proposition 1.4(1)]{Fe6}).
\vspace{.001cm}

\noindent {\bf Example.} Let $W(1;2)=\F e_{-1}\oplus\F e_0\oplus\F e_1\oplus\F e_2$
 be a Zassenhaus algebra over an algebraically closed field $\F$ of characteristic two.
(Here we set $e_{-1}:=\partial$ and $e_k:=x^{(k+1)}\partial$ for $k=0,1,2$, where
$\caO(1;2):=\F x^{(0)}\oplus\F x^{(1)}\oplus\F x^{(2)}\oplus\F x^{(3)}$ is a divided
power algebra in one variable $x$ over $\F$, $\partial(x^{(0)}):=0$, and $\partial(x^{(a)})
:=x^{(a-1)}$ for $1\le a\le 3$.) Then $W(1;2)^{(1)}=\F e_{-1}\oplus\F e_0\oplus\F e_1$
is a simple Lie algebra and $\fW:=\fW(1;2)=\F e_{-2}\oplus\F e_{-1}\oplus\F e_0\oplus
\F e_1\oplus\F e_2$ (with $e_{-2}:=\partial^2$) is a minimal $2$-envelope of $W(1;2)$
as well as of $W(1;2)^{(1)}$. Note that $\fW$ is the full derivation algebra of $W(1;2)^{(1)}$
(see \cite[Proposition~3.3(1)]{Str3}), and as such $\fW$ is a restricted Lie algebra.
The elements $t_\pm:=e_0+e_{\pm 1}+e_{\pm 2}$ are toral, and $\fT:=\F t_+\oplus
\F t_-$ is a two-dimensional torus of $\fW$. It follows from \cite[Theorem 7.6.3(2) and
Lemma~1.2.6(1)]{StrI} that $\fT$ is a torus of maximal dimension.

Set $\fB:=\F e_0\oplus\F e_1\oplus\F e_2$. Since $\fB$ is a $2$-subalgebra of $\fW$
with one-dimensional torus $\F e_0$ and two-dimensional $2$-unipotent radical $\F e_1
\oplus\F e_2$, $\fB$ has two isomorphism classes of irreducible restricted modules, which
can be represented by $F_\lambda:=\F1_\lambda$ ($\lambda=0,1$), where $e_0\cdot
1_\lambda:=\lambda 1_\lambda$ and $e_k\cdot 1_\lambda:=0$ for $k=1,2$ (cf.\ \cite[Lemma
2.4]{VK} or \cite[Lemma 5.8.6(2)]{SF}).

Consider the restricted baby Verma modules $Z(\lambda):=\ind_\fB^\fW(F_\lambda,0)$
for $\lambda=0,1$, and let $S$ denote any irreducible restricted $\fW$-module. Then
it follows from Frobenius reciprocity that $\Hom_\fW(Z(\lambda),S)\cong\Hom_\fB
(F_\lambda,S_{\vert\fB})\neq 0$ for some $\lambda=0$ or $\lambda=1$, and therefore
$S$ is a homomorphic image of $Z(\lambda)$ for $\lambda=0$ or $\lambda=1$. On the
other hand, one concludes from \cite[Corollary 1.5 and Lemma 1.6]{Fe} that $Z(\lambda)$
has a unique maximal submodule, and therefore $Z(\lambda)$ has a unique irreducible
factor module. Hence $\fW$ has at most two isomorphism classes of irreducible restricted
modules.

It is a consequence of \cite[Proposition 1.4(1)]{Fe6} and Lemma \ref{psim}(3) that
$W(1;2)^{(1)}=[\fW,\fW]$ is a non-trivial irreducible $\fW$-module. So $\fW$ has exactly
two isomorphism classes of irreducible restricted modules, and up to isomorphism $W
(1;2)^{(1)}$ is the only non-trivial irreducible restricted $\fW$-module. Finally, it follows
from $[W(1;2)^{(1)}]^\fT=0$ and Lemma \ref{inv} that $\fW$ has maximal $0$-PIM,
but $\fW$ is not solvable as it contains the simple subalgebra $W(1;2)^{(1)}$.
\vspace{.3cm}

In particular, by the argument in the proof of Theorem \ref{numirr} one obtains that $\vert
\irr(\fW,\chi)\vert\le 4$ for every  linear form $\chi$ on $\fW$, but one would expect
even that $\vert\irr(\fW,\chi)\vert\le 2$ as in the restricted case. Moreover, it is not
difficult to show that the projective cover of the non-trivial irreducible restricted
$\fW$-module $W(1;2)^{(1)}=[\fW,\fW]$ has the same dimension as the projective
cover of the one-dimensional trivial $\fW$-module. We will leave the details to the
interested reader and will investigate this more generally for the minimal $2$-envelope
of $W(1;n)$ on another occasion.
\vspace{.3cm}

\noindent {\bf Acknowledgements.} The first and the second author would like to
thank the Dipartimento di Matematica e Applicazioni at the Universit\`a degli Studi
di Milano-Bicocca for the hospitality during their visit in May 2013. The third author
would like to thank the Department of Mathematics and Statistics at the University
of South Alabama for the hospitality during his visit in October 2013.

%%%%%%%%%%%%%%%%%%%%%%%%%%%%%%%%%%%%%%%%%%%%%%%%%%%%%%%%%%%%%%%%%%%%%%%%%%%%%%%%%%%%%%%%%%%%%%%%%%%  

%%%%%%%%%%%%%%%%%%%%%%%%%%%%%%%%%%%%%%%%%%%%%%%%%%%%%%%%%%%%%%%%%%%%%%%%%%%%%%%%%%%%%%%%%%%%%%%%%%%

\end{document}